\documentclass[11pt]{article}
\usepackage[utf8]{inputenc}
\usepackage{fullpage}
\usepackage{tocloft}

\usepackage[square,sort,comma,numbers]{natbib}

\usepackage{amsthm}
\usepackage{amssymb}
\usepackage{amsmath}
\usepackage{pdfpages}
\usepackage{booktabs}
\usepackage{graphics}
\usepackage{tikz,lipsum,lmodern}
\usepackage{sidecap}
\sidecaptionvpos{figure}{c}
\usepackage[most]{tcolorbox}
\colorlet{shadecolor}{gray!40}
\usepackage[bookmarksnumbered=true]{hyperref} 
\hypersetup{
     colorlinks = true,
     linkcolor = blue,
     anchorcolor = blue,
     citecolor = blue,
     filecolor = blue,
     urlcolor = blue
     }
\usetikzlibrary{patterns}

\theoremstyle{plain}

\newtheorem{theorem}{Theorem}
\newtheorem{proposition}{Proposition}
\newtheorem{corollary}{Corollary}
\newtheorem{lemma}{Lemma}

\theoremstyle{definition}

\newtheorem{definition}{Definition}
\newtheorem{example}{Example}

\theoremstyle{remark}

\newtheorem{remark}{Remark}

\newcommand{\id}{\text{id}}
\newcommand{\ini}{\text{in}}

\def\NZQ{\mathbb}               

\def\CC{{\NZQ C}}

\def\frk{\mathfrak}               

\def\a{{\bf a}}

\def\Phi{{\frk n}}
\def\Phi{{\frk N}}

\def\A{{\mathcal A}}

\def\opn#1#2{\def#1{\operatorname{#2}}} 
\opn\chara{char} \opn\length{\ell} \opn\pd{pd} \opn\rk{rk}
\opn\projdim{proj\,dim} \opn\injdim{inj\,dim} \opn\rank{rank}
\opn\depth{depth} \opn\grade{grade} \opn\height{height}
\opn\embdim{emb\,dim} \opn\codim{codim}

\opn\Tr{Tr} \opn\bigrank{big\,rank}
\opn\superheight{superheight}\opn\lcm{lcm}
\opn\trdeg{tr\,deg}
\opn\reg{reg} \opn\lreg{lreg} \opn\ini{in} \opn\lpd{lpd}
\opn\size{size} \opn\sdepth{sdepth}
\opn\link{link}\opn\fdepth{fdepth}\opn\lex{lex}
\opn\LM{LM}
\opn\LC{LC}
\opn\NF{NF}
\opn\Merge{Merge}
\opn\sgn{sgn}
\opn\Diag{Diag}

\opn\div{div} \opn\Div{Div} \opn\cl{cl} \opn\Pic{Pic}
\opn\Prin{Prin}
\opn\op{op}
\opn\indeg{indeg} \opn\outdeg{outdeg}
\opn\Spec{Spec} \opn\Supp{Supp} \opn\supp{supp} \opn\Sing{Sing}
\opn\Ass{Ass} \opn\Min{Min}\opn\Mon{Mon} \opn\val{val}
\opn\Ann{Ann} \opn\Rad{Rad} \opn\Soc{Soc}
 \opn\Ker{Ker} \opn\Coker{Coker} \opn\Am{Am}
\opn\Hom{Hom} \opn\Tor{Tor} \opn\Ext{Ext} \opn\End{End}
\opn\Aut{Aut} \opn\id{id}

\opn\nat{nat}
\opn\pff{pf}
\opn\Pf{Pf} \opn\GL{GL} \opn\SL{SL} \opn\mod{mod} \opn\ord{ord}
\opn\Gin{Gin} \opn\Hilb{Hilb}\opn\sort{sort}
\opn\Image{Image}
\opn\vol{Vol}
\opn\aff{aff} \opn\con{conv} \opn\relint{relint} \opn\st{st}
\opn\lk{lk} \opn\cn{cn} \opn\core{core} \opn\vol{vol}
\opn\link{link} \opn\star{star}\opn\lex{lex}\opn\set{set}
\opn\dist{dist}
\opn\gr{gr}

\opn\dis{dis}
\def\pnt{{\raise0.5mm\hbox{\large\bf.}}}

\opn\Lex{Lex}
\opn\syz{{\rm syz}}
\opn\spoly{{\rm spoly}}
\opn\LM{{\rm LM}}
\opn\lm{{\rm lm}}
\opn\type{{\rm type}}
\opn\lcm{{\rm lcm}} \opn\A{\mathcal A}

\def\pot#1#2{#1[\kern-0.28ex[#2]\kern-0.28ex]}

\opn\dirlim{\underrightarrow{\lim}}
\opn\inivlim{\underleftarrow{\lim}}
\tikzstyle{Cgray}=[scale = .45,circle, fill = gray, minimum size=3mm] \tikzstyle{Cblack}=[scale = .7,circle, fill = black, minimum size=3mm]
\tikzstyle{Cblue}=[scale = .5,circle, fill = blue, inner sep = 0pt, minimum
size=3mm]
\tikzstyle{C1}=[scale = .7,circle, fill = black!0, inner sep = 0pt, minimum
size=3mm]

\tikzstyle{test2}=[scale = 1.5,circle, fill = black!0, inner sep = 0pt, minimum
size=3mm]
\tikzstyle{Cwhite}=[scale = .45,circle, fill = white, minimum size=3mm] 
\tikzstyle{Cblack2}=[scale = .3,circle, fill = black, minimum size=3mm] 
\tikzstyle{Cblack}=[scale = .3,circle, fill = black, minimum size=3mm]
\tikzstyle{C0}=[scale = .9,circle, fill = black!0, inner sep = 0pt, minimum size=3mm]
\tikzstyle{C1}=[scale = .7,circle, fill = black!0, inner sep = 0pt, minimum size=3mm]
\tikzstyle{Cred}=[scale = .4,circle, fill = red, minimum size=3mm] 
\tikzstyle{Cblack3}=[scale = .4,circle, fill = black, minimum size=3mm]

\title{Minimal cellular resolutions of powers of matching field ideals}
\author{Oliver Clarke and Fatemeh Mohammadi}
\begin{document}
\maketitle

\abstract{We study a family of monomial ideals, called block diagonal matching field ideals, which arise as monomial Gröbner degenerations of determinantal ideals. Our focus is on the minimal free resolutions of these ideals and all of their powers. Initially, we establish their linear quotient property and compute their Betti numbers, illustrating that their minimal free resolution is supported on a regular CW complex. Our proof relies on the results of Herzog and Takayama, demonstrating that ideals with a linear quotient property have a minimal free resolution, and on the construction by Dochtermann and Mohammadi of cellular realizations of these resolutions. We begin by proving the linear quotient property for each power of such an ideal. Subsequently, we show that their corresponding decomposition map is regular, resulting in a minimal cellular resolution. Finally, we demonstrate that distinct decomposition maps lead to different cellular complexes with the same face numbers.

{\hypersetup{linkcolor=black}
\setcounter{tocdepth}{1}
\setlength\cftbeforesecskip{1pt}
{\tableofcontents}}

\section{Introduction}
Consider \( R = \mathbb{K}[x_1, x_2, \ldots, x_n] \), the graded polynomial ring over a field \( \mathbb{K} \), and \( I \subset R \) an ideal generated by monomials \( I = \langle m_1, m_2, \ldots, m_k \rangle \). Investigating minimal graded resolutions of such monomial ideals is a common endeavor in combinatorial commutative algebra. In their work \cite{HT}, Herzog and Takayama introduced a general technique called the ``mapping cone" for constructing a free resolution of the monomial ideal \( I \) under the condition that the ideal exhibits a linear quotient property. The core concept involves the gradual inclusion of generators \( m_i \) one by one, systematically constructing the resolution of the ideal \( \langle m_1, \ldots, m_i \rangle \) as a mapping cone derived from a suitable map between previously constructed complexes for the ideal \( \langle m_1, \ldots, m_{i-1} \rangle \). Various families of ideals stemming from graphs or matroids have been shown to possess the linear quotient property \cite{mohammadi2010weakly, mohammadi2011powers}.

\medskip

In \cite{batzies2002discrete}, Batzies and Welker utilized Discrete Morse theory to construct a cellular resolutions  of ideals featuring linear quotients (or the shellable property). Later, in \cite{dochtermann2012cellular}, Dochtermann and Mohammadi provided an explicit construction of such complexes. More precisely, they showed that when certain conditions are met, characterized by the presence of a regular decomposition function, the minimal free resolution of any ideal with a linear quotient property has a cellular structure. That is, there exists a CW complex whose vertices symbolize the generators of 
$I$ and higher-dimensional faces representing the higher-dimensional syzygies in the resolution. This idea is introduced in the seminal paper  \cite{BS98} by Bayer and Sturmfels. While not all minimal free resolutions of monomial ideals are supported on a CW-complex \cite{velasco2008minimal}, an interesting question emerges: which algebraic minimal free resolutions can indeed be realized cellularly, accompanied by a geometric or combinatorial characterization?
It has been established that several well-studied classes of monomial ideals have cellular resolutions. For example, Novik, Postnikov, and Sturmfels, in \cite{novik2002syzygies}, constructed minimal cellular resolutions for monomial ideals associated to (orinted) matroids. Furthermore, Mermin \cite{mermin2010eliaiiou} showed that the Eliahou-Kervaire resolution of a stable ideal 
is supported on a CW complex. Sinefakopoulos \cite{sinefakopoulos2008borel} established this property for Borel fixed ideals, while Iglesias and Sáenz-de-Cabezón recently extended this result to quasi-stable ideals \cite{iglesias2022cellular}. Engström and Dochtermann \cite{dochtermann2012cellular} showed this characteristic for family of co-interval ideals. In particular, the ideals arising from graphs, simplicial complexes and posets, and their cellular resolutions have been extensively studied; see e.g.~\cite{mohammadi2014powers,mohammadi2014divisors,mohammadi2016divisors,mohammadi2016reliability, ene2011monomial, mohammadi2018prime,faridi2024cellular}. 
\medskip

Once an ideal possesses a cellular resolution, a natural question arises: can this property be extended to its powers? 
{Investigation into this problem began with Bertram, Ein, and Lazarsfeld \cite{bertram1991vanishing}. However, exploring this avenue is a highly challenging task and has only been undertaken for very specific ideals. For instance, the question has been explored for powers of maximal ideals \cite{almousa2022polarizations, mermin2010eliaiiou}, powers of path ideals \cite{engstrom2012cellular}, and powers of square-free
monomial ideals of projective dimension one~\cite{cooper2022powers}.} There are even multiple examples where the power of an ideal with a linear resolution fails to have a linear resolution; see e.g.~\cite{conca2005regularity} and the references therein.

\medskip
Here, we study this problem for the family of matching field ideals. 
More precisely, we study a specific family of monomial degenerations obtained as matching field ideals, as defined by Sturmfels and Zelevinsky \cite{sturmfels1993maximal}, which can also be obtained as weight degenerations \cite{Fatemeh2024}.
We establish that all members of this family, and all their powers,  exhibit the linear quotient property, following the definition by Herzog and Takayama \cite{HT}, and satisfy additional conditions characterized by Dochtermann and Mohammadi~\cite{FatemehAnton}. Consequently, these ideals have cellular resolutions.

\medskip
Consider a generic \(3 \times n\) matrix of indeterminants \(X = (x_{ij})\) and the ideal generated by maximal minors of \(X\). The minimal free resolutions of this ideal \cite{bruns2006determinantal, bruns2022determinants}, as well as its powers \cite{bruns2015maximal}, have been extensively studied in commutative algebra. In particular, it has been shown that they possess linear resolutions. Here, we investigate certain monomial degenerations of these ideals, also known as initial ideals, and strengthen these results by demonstrating that all these ideals and their powers have linear resolutions, and more importantly, their resolutions are cellular. 
More precisely, we establish the linear quotient property for all powers of the block diagonal matching field ideals. Furthermore, we demonstrate that the minimal free resolution of each such ideal is supported on a CW complex. While this property has been previously established for the matching field ideals themselves~\cite{Fatemeh2024}, our work extends it to their powers.
\begin{example}
    Fix $n \ge 3$, and $X = (x_1, \dots, x_n; y_1, \dots, y_n; z_1, \dots, z_n)$ a $3\times n$ matrix of indeterminates. The first examples of matching field ideals are the diagonal matching field ideals, which are generated by the monomials $x_i y_j z_k$ for $1 \leq i < j < k\leq n$. The homological invariants of these ideal arising from the diagonal case are known; see~Corollary~\ref{cor:cm}. Moreover, these ideals are all square-free and share the same Betti numbers as long as $n$ is fixed; however, they have different resolutions and distinct cellular resolutions. In fact, their minimal free resolutions are supported on different CW complexes, with the associated Betti numbers remaining invariant for a fixed $n$; see~\cite{Fatemeh2024}.
    For \( n = 4 \), we have three different matching field ideals, as follows:
\begin{itemize}
    \item \( M_{(4)} = \langle x_1 y_2 z_3, x_1 y_2 z_4, x_1 y_3 z_4, x_2 y_3 z_4 \rangle \)
\item \( M_{(2,2)} = \langle x_1 y_2 z_3, x_1 y_2 z_4, x_3 y_1 z_4, x_3 y_2 z_4 \rangle \)
\item \( M_{(1,3)} = \langle x_2 y_1 z_3, x_2 y_1 z_4, x_3 y_1 z_4, x_2 y_3 z_4 \rangle \)
\end{itemize}
The Betti numbers of all these ideals are identical, with \(\beta_0 = 4\) and \(\beta_1 = 3\). 
\end{example}
\noindent{\bf Outline.}~In Section~\ref{sec:pre}, we revisit the concepts of linear quotients and matching field ideals. Section~\ref{sec:linear} is devoted to proving the linear quotient property for all powers of the block diagonal matching field ideals. Subsequently, in Section~\ref{sec:cellular}, we demonstrate that the minimal free resolution of each such ideal is supported on a regular CW complex. Finally, we provide illustrative examples of cellular resolutions by depicting their associated CW complexes.

\medskip\noindent{\bf Acknowledgement.} We would like to thank Lara Bossinger, Anna Brosowsky, Janet Page, and Alexandra Seceleanu for their comments on a preliminary version of the paper and pointing out the example in Remark~\ref{rmk: no canonical rep tableau}.
F.M. is partially supported by the FWO grants G0F5921N (Odysseus programme) and G023721N, the UiT Aurora project MASCOT and KU Leuven grant iBOF/23/064.

\section{Preliminaries}\label{sec:pre}

\subsection{Ideals with linear quotient property}\label{Mapcones}

We begin by recalling the ideals with linear quotients property from \cite{HT}.  Fix a field \( \mathbb{K} \) and consider the polynomial \( R = \mathbb{K}[x_1, \ldots , x_n] \) over $\mathbb{K}$ in \( n \) variables.
\begin{definition}
Consider a monomial ideal $I\subset R$ with the minimal monomial generating set $G(I)=\{m_1,\ldots,m_k\}$. The ideal $I$ is said to have {\it linear quotients} if there exists an ordering of its generators $m_1, \dots, m_k$ such that for each $j$ the colon ideal $\langle m_1,\ldots,m_{j-1}\rangle : m_j$ is generated by a subset of the variables $x_1,\ldots,x_n$. We denote this subset of variables by set$(m_j)$.
\end{definition}

It is shown in \cite{HT} that for such ideals the mapping cone construction produces a minimal free resolution, and the Betti numbers of $I$ can be explicitly described as 
$
\beta_\ell(I)=\sum_{j=1}^k{\binom{|{\rm set}(m_j)|}{\ell-1}}.
$ In particular, Herzog and Takayama provided an explicit description of the \emph{differentials} in resolutions for a certain class of these ideals. More precisely, let $M(I)$ denote the set of all monomials in $I$, and $G(I)$ the minimal generating set of $I$. We define the \emph{decomposition function of $I$} to be the assignment $b: M(I)\rightarrow G(I)$ such that $b(m) = m_j$, where $j$ is the smallest number such that $m\in \langle m_1,\ldots,m_{j}\rangle$.  
We consider the following family of decomposition functions.
\begin{definition} \label{def:regular}
A decomposition function $b$ of $I$ is called \emph{regular}, if for each $m\in G(I)$ and every
$t\in \set(m)$ we have
\[
 \quad\set(b(x_t m))\subseteq \set(m).
\]
In this case, for each $m\in G(I)$ and $t,s\in \set(m)$
 we have the {\em exchange property}:
\[
(*)\quad \quad b(x_s b(x_t m))=b(x_t b(x_s m)).
\]
Additionally, if $x_t \notin \text{set}(b(x_s m))$, then $b(x_t b(x_s m))$ is defined to be $b(x_s m)$. In fact, for every element $m$ and each variable $x_t \notin \text{set}(m)$, we have $b(x_t m) = m$.
\end{definition}

\begin{remark}
The term {\it regular decomposition function} originates from the fact that the corresponding minimal free resolution of the ideal \(I = \langle m_1, \ldots, m_r\rangle \) can be constructed by gluing the Koszul complex associated with the colon ideal \(\langle m_1, \ldots, m_{r-1}\rangle : m_r \) to the resolution of the previous ideal \( \langle m_1, \ldots, m_{r-1}\rangle \) in an iterative fashion. This process is reminiscent of the concept of a regular CW complex, which is defined as a CW complex in which the gluing maps are homeomorphisms.
\end{remark}

In \cite{FatemehAnton}, Dochtermann and Mohammadi constructed a minimal cellular resolution of ideals with linear quotients property and regular decomposition map, as follows. 
\begin{theorem}
\label{Theorem_cone}
Suppose $I$ has linear quotients with respect to some ordering $m_1, \dots, m_k$ of the generators, and suppose that $I$ has a regular decomposition function.  Then the minimal resolution of $I$ obtained as an iterated mapping cone is cellular and supported on a regular $CW$-complex.
\end{theorem}

\begin{remark}
We note that in the proofs of Theorem \ref{Theorem_cone} and the related lemmas, the only property of the decomposition function that is required is the exchange property ($*$), and not its definition in terms of assigning a particular generator to a monomial or that $\set(b(x_tm))\subseteq \set(m)$. 
\end{remark}

\subsection{Block diagonal matching field ideals $M_\a$}
Fix $n \ge 3$, and $X = (x_1, \dots, x_n; y_1, \dots, y_n; z_1, \dots, z_n)$ a $3\times n$ matrix of indeterminates. Let $F \subseteq \CC[X]$ be the set of maximal minors of $X$. A $(3,n)$ matching field is a combinatorial object that picks out a single term of $f$, for each $f \in F$. Concretely, a matching field is uniquely determined by specifying a monomial $x_i y_j z_k$ for each $3$-subset $\{i,\,j,\,k\} \subseteq [n]$. For more background, we refer to reader to \cite{sturmfels1993maximal,bernstein1993combinatorics}. The \textit{monomial matching field ideal}, or simply the \textit{matching field ideal}, is the monomial ideal generated by the terms distinguished by the matching field. For example, the \textit{diagonal matching field}, respectively the \textit{anti-diagonal matching field}, selects the leading-diagonal term, respectively anti-diagonal term, of each maximal minor. These (anti)-diagonal terms generate the initial ideal of $\langle F \rangle$ with respect to any term order such that the leading term of each $f \in F$ is its (anti)-diagonal term \cite{CCA}. We now recall the definition of a block diagonal matching field ideal from \cite{KristinFatemeh}, which generalises the diagonal matching field.

\medskip

\noindent
Fix a sequence of positive integers $ {\bf a} = (a_1, a_2, \dots,  a_{r})$ so that $\sum_{i = 1}^r a_i = n$. We say that $\a$ is a partition of $[n]$. For $1 \leq s\leq r$, let 
$B_s = \{\alpha_{s-1}+1, \alpha_{s-1} +2, \dots, \alpha_{s}\}$,
where $\alpha_{s} = \sum_{i = 1}^{s} a_{i }$ and $\alpha_0 = 0$. 
We write $B_\a = (B_1 | \cdots |B_r)$ for the ordered partition of $[n]$ into sets $B_i$. We refer to the sets $B_i$ as the blocks of the partition $\a$.

\begin{definition}\label{def:blockdiagcolumns}
Fix a partition $\a$ of $[n]$ with blocks
$B_\a=(B_1|\cdots|B_r)$, as above. For every $3$-subset $\{i<j<k\}$ of $[n]$, let $s \in [r]$ be the smallest value such that $B_s \cap \{i,j,k\} \neq \emptyset$. Then we associate to $\{i,j,k\}$, the monomial
\[
m_{\{ijk\}} =
\left\{
	\begin{array}{ll}
		x_jy_iz_k & \mbox{if } |\{i,j,k\}\cap B_s|=1 \\
		x_iy_jz_k & \mbox{otherwise}.
	\end{array}
\right.
\]
The block diagonal matching field ideal $M_\a$ associated to the partition $\a$ is the ideal generated by the monomials $m_{\{i,j,k\}}$ associated to all 3-subsets $\{i,j,k\}$ of [n].
\end{definition}

Observe that, for each block $B_i$ of $\a$, the restriction of the generators of $M_\a$ to those involving only variables $x_j, y_j, z_j$, with $j \in B_i$, yields a diagonal matching field. Hence the name: \textit{block-diagonal matching field ideal}. For instance, if $\a = (n)$, then $m_{i<j<k} = x_i y_j z_k$ is the leading diagonal term of each $3$-minor of $X$.

Following conventions of \cite[Section 2]{KristinFatemeh}, a block diagonal matching field ideal $M_\a$ induces an ordering on each $3$-subset of $[n]$ as follows. If $A$ is a $3$-subset and $m_A = x_i y_j z_k$, then the ordering on the elements of $A$ according to the matching field is $(i, j, k)$.

\begin{example}
    The block diagonal matching field $B_{2,3}$ has blocks $B_1 = 12$ and $B_2 = 345$. The matching field ideal $M_{2,3} \subseteq \mathbb K[x_1, \dots, x_5, y_1, \dots, y_5, z_1, \dots, z_5]$ is generated by the monomials:
    \begin{multline*}
        m_{123} = x_1y_2z_3,\ 
        m_{124} = x_1y_2z_4,\ 
        m_{125} = x_1y_2z_5,\ 
        m_{134} = x_3y_1z_4,\ 
        m_{135} = x_3y_1z_5,\\
        m_{145} = x_4y_1z_5,\ 
        m_{234} = x_3y_2z_4,\ 
        m_{235} = x_3y_2z_5,\ 
        m_{245} = x_4y_2z_5,\ 
        m_{345} = x_3y_4z_5.
    \end{multline*}
    So the ordering of the elements of the sets $123$, $134$, $345$ according to the matching field is given by $(1,2,3)$, $(3,1,4)$, $(3,4,5)$.
\end{example}
\color{black}

Consider the diagonal matching field $B_{(n)}$ and its associated ideal $M_{(n)}=\langle x_iy_jz_k:\ 1\leq i<j<k\leq n\rangle$ in the ring $R=\mathbb{C}[x_i,y_i,z_i:\ i=1,\ldots,n]$. The homological invariants of the ideal $M_{(n)}$ has been extensively studied \cite{herzog1992grobner, bruns1992computation, sturmfels1990grobner}. The following corollary is a special case of \cite[Corollary~1.3]{ene2013determinantal}. For further reference, see also \cite[Theorem 6.35]{ene2011grobner}, \cite[Corollary 3.3.5]{Herzog2011}, \cite[Theorem 6.9]{bruns1992computation}, and \cite[Theorem 3.5]{herzog1992grobner}.
\begin{corollary}
\label{cor:cm}
Consider the diagonal matching field $B_{(n)}$ and its associated ideal $M_{(n)}=\langle x_iy_jz_k:\ 1\leq i<j<k\leq n\rangle$ in the ring $R=\mathbb{C}[x_i,y_i,z_i:\ i=1,\ldots,n]$. Then:
\begin{itemize}
	\item[{\rm (a)}]
	$\height M_{(n)}=\projdim{M_{(n)}}=n-2$.
	\item [{\rm (b)}] $M_{(n)}$ is Cohen-Macaulay.
	\item [{\rm (c)}] The Hilbert series of $R/M_{(n)}$ has the form
	$H_{R/M_{(n)}}(t)=\frac{Q(t)}{(1-t)^{2n-2}}$,
	where
	\[
	Q(t)=\left[\det\left(\sum_k{\binom{3-i}{k}}{\binom{n-j}{k}}\right)_{1\leq i,j\leq 2}\right]/ t.
	\]	
	\item [{\rm (d)}] The multiplicity of $R/M_{(n)}$ is
	$e(R/M_{(n)})={\binom{n}{2}}.$
\end{itemize}
\end{corollary}
The homological invariants of diagonal matching field ideals, as discussed above, are well-studied. However, the invariants of their powers remain largely unexplored. In particular, the main homological invariant we will focus on in this note is the Betti numbers and the minimal free resolution of these ideals and their powers.

\section{Powers of matching field ideals have linear quotients}\label{sec:linear}
Fix a partition $\a=(a_1,\ldots,a_r)$ of $[n]$ and a positive integer $\ell$. Let $M_\a$ be the corresponding block diagonal matching field ideal as in Definition~\ref{def:blockdiagcolumns}. In this section, we introduce an ordering on the set of tableaux, whose columns correspond to generators of $M_\a$. We will subsequently use this ordering to give an ordering of the generators of $(M_\a)^\ell$ to prove that the ideal has the linear quotient property.

Let $(B_1 | B_2 | \dots | B_r)$ denote the intervals of the partition $\a$ of $[n]$. For a given $3$-subset $A$ of $[n]$, if the order of the elements in $A$ according to the matching field is $(i, j, k)$, we define the type of $A$ as $\text{type}(A) = (I, J, K)$, where $I, J$, and $K$ are integers such that $i \in B_I$, $j \in B_J$, and $k \in B_K$. In general, lowercase letters will refer to entries of the tableau, while their uppercase counterparts will denote their respective types. Note that, by the definition of the matching field, we have that either $i < j < k$ or $j < i < k$. The latter case happens if and only if $J < I$.

\subsection{Representations of the elements of $(M_\a)^\ell$}\label{subsec:canonical}
Fix a $3$-subset $A = \{i,j,k\}$ whose corresponding ordered set is $(i, j, k)$ in the matching field with $\text{type}(A) = (I, J, K)$. If $i < j < k$ then we have $I = J$, otherwise if $j < i < k$ then we have $J < I$. So we always have that $J \le I$. 
We associate to $A$ a column vector, denoted by $P_A$, with entries $x_i$ in the first row, $y_j$ in the second row, and $z_k$ in the third row. In other words, the product of the entries is the monomial $m_A$ in $M_\a$. A \textit{tableau} $T$ is a concatenation of columns $T = P_{A_1} P_{A_2} \dots P_{A_\ell}$. The monomial corresponding to $T$ is the product of all variables in every column. So $T$ is a $3 \times \ell$ tableau whose entries are the variables such that product of entries in column $i$ is the monomial $m_{A_i}$. Note that there may exist a different tableau $T'$ with the same monomial as $T$. In this case, the tableaux $T$ and $T'$ have the same row-wise contents, and we say that $T'$ is a representation of $T$.

\begin{lemma}[Simplification of tableaux]\label{lem:rep}
Let $T = P_{A_1},\ldots,P_{A_\ell}$ be the tableau corresponding to the $3$-subsets $A_1, A_2, \ldots, A_\ell$ ordered by the matching field. 
Let ${\rm type}(A_u)=(I_u,J_u,K_u)$ and $k_u = \max(A_u)$ for each $u \in [\ell]$.
Then there is a representation of $T$ as $T'=P_{A'_1},P_{A'_2},\ldots,P_{A'_\ell}$ such that:
\begin{itemize}
    \item $I'_1\leq I'_2\leq\cdots\leq I'_\ell$,
    \item $k'_1\leq k'_2\leq\cdots\leq k'_\ell$,
\end{itemize}
where $\type(A'_u) = (I'_u, J'_u, K'_i)$ and $k'_u = \max(A'_u)$ for each $u$. 
\end{lemma}

\begin{proof}[{\bf Proof of Lemma~\ref{lem:rep}}]
Without loss of generality we may assume that $k_1 \le k_2 \le \dots \le k_\ell$ by permuting the order of the columns of $T$. Let $s$ be the largest value such that,
\[
I_1 \le I_2 \le \dots \le I_{s}, 
\]
and $I_1, \dots, I_{s}$ are the $s$ smallest values in the multiset $\{I_1, \dots, I_\ell \}$. Explicitly if $I_1, \dots, I_\ell$ are ordered by $I_{i_1} \le \dots \le I_{i_\ell}$, then we have $I_1 = I_{i_1}, I_2 = I_{i_2}, \dots, I_s = I_{i_s}$. We show that it is possible to reorder the entries of $T$ such that $I_1 \le I_2 \le \dots \le I_\ell$ and $k_1 \le k_2 \le \dots \le k_\ell$.

We proceed by reverse induction on $s$. If $s = n$, then the result holds. Assume that $s < n$ and let, $I_t = \min \{I_{s + 2}, \dots, I_\ell \} \text{ such that } I_t < I_{s+1}$.
So we have, $I_1 \le I_2 \le \dots \le I_{s} \le I_t < I_{s + 1}$. Consider columns $t$ and $s + 1$ of $T$,
\[
\begin{tabular}{|c|c|}
    \multicolumn{1}{c}{$t$} & \multicolumn{1}{c}{$s + 1$} \\
    \hline
    $i_t$ & $i_{s + 1}$ \\
    \hline
    $j_t$ & $i_{s + 1}$ \\
    \hline
    $k_t$ & $k_{s + 1}$ \\
    \hline
\end{tabular}
\text{ and their types }
\begin{tabular}{|c|c|}
    \multicolumn{1}{c}{$t$} & \multicolumn{1}{c}{$s + 1$} \\
    \hline
    $I_t$ & $I_{s + 1}$ \\
    \hline
    $J_t$ & $J_{s + 1}$ \\
    \hline
    $K_t$ & $K_{s + 1}$ \\
    \hline
\end{tabular}\, .
\]
We have that the following inequalities:
\[
J_t \le I_t \le K_t,\ \ J_{s + 1}\le I_{s + 1}\le K_{s + 1},\ \ I_t < I_{s + 1},\ \ k_{s + 1} \le k_{t}.
\]    
To complete the proof, we show that we can replace columns $t$ and $s + 1$ of $T$ with the columns $t'$ and $(s + 1)'$ below, respectively
\[
\begin{tabular}{|c|c|}
    \multicolumn{1}{c}{$(s + 1)'$} & \multicolumn{1}{c}{$t'$} \\
    \hline
    $i_t$ & $i_{s + 1}$ \\
    \hline
    $j_t$ & $j_{s + 1}$ \\
    \hline
    $k_{s + 1}$ & $k_{t}$ \\
    \hline
\end{tabular} \, .
\]
To do this, it is enough to show that columns $(s + 1)'$ and $t'$ are valid, i.e., their ordering respects the matching field. First, the column $(s + 1)'$ is valid because $J_t \le I_t$ and $I_t < I_{s + 1} \le K_{s + 1}$. Second, the column $t'$ is valid as $J_{s + 1} \le I_{s + 1}$ and $k_{s + 1} \le k_t$, which gives $I_{s + 1} \le K_{s + 1} \le K_{t}$.
\end{proof}

\begin{remark}
It is also possible to give a direct algorithmic of Lemma~\ref{lem:rep} as follows.
We first show that Lemma~\ref{lem:rep} holds when $T$ is a tableau with two columns. The result gives us
\[
\text{type}(T) = 
\begin{tabular}{|c|c|}
     \hline
    $I_1$ & $I_2$ \\
    \hline
    $J_1$ & $J_2$ \\
    \hline
    $K_1$ & $K_2$ \\
    \hline
\end{tabular} \, , \quad
\text{type}(T') = 
\begin{tabular}{|c|c|}
     \hline
    $\min\{I_1, I_2\}$ & $\max\{I_1, I_2\}$ \\
    \hline
    $J_1/J_2$ & $J_1/J_2$ \\
    \hline
    $\min\{K_1, K_2\}$ & $\max\{K_1, K_2\}$ \\
    \hline
\end{tabular}
\]
where there is no condition on the order of the second row.
To see that this operation is valid, observe that $J_1 \le I_1 \le K_1$ and $J_2 \le I_2 \le K_2$. It follows that $\min\{I_1, I_2\} \le \min\{K_1, K_2\}$ and $\max\{I_1, I_2\} \le \max\{K_1, K_2\}$. Therefore $\text{type}(T')$ is a valid tableau of types and $T$ may be represented in this way. The same argument holds if we replace $K_i$ with the `true value' $k_i$ of the tableau in that position.

Now, given a tableau $T$ with columns $P_{A_1}, P_{A_2}, \dots, P_{A_{\ell}}$, we can construct a representation of $T$ in the desired form by applying the two-column construction described above to the following sequence of pairs of columns in $T$:
$
(1,2), (1,3), \dots, (1,\ell), (2,3), \dots, (2,\ell), \dots, (\ell-1, \ell).$
\end{remark}

\begin{remark}\label{rmk: no canonical rep tableau}
    It is not always possible to find a representation $T'$ of $T$ with $J'_1 \le J'_2 \le \dots \le J'_\ell$. For instance, fix the partition $(1 | 234 | 5)$. Consider the tableau
    $T = (2, 4; 3, 1; 4, 5)$ with $\text{type}(T) = (2, 2; 2, 1; 2, 3)$. However, none of the tableaux $T'$ with $\text{type}(T') = (2,2; 1,2; 2,3)$ that are row-wise equal to $T$ are valid:
    \[
    T = \begin{tabular}{|c|c|}
        \hline
        $2$ & $4$ \\
        \hline
        $3$ & $1$ \\
        \hline
        $4$ & $5$ \\
        \hline
    \end{tabular}\,, 
    \text{ but }
    T' = \begin{tabular}{|c|c|}
        \hline
        $2$ & $4$ \\
        \hline
        $1$ & $3$ \\
        \hline
        $4$ & $5$ \\
        \hline
    \end{tabular} 
    \text{ and }
    T' = \begin{tabular}{|c|c|}
        \hline
        $4$ & $2$ \\
        \hline
        $1$ & $3$ \\
        \hline
        $4$ & $5$ \\
        \hline
    \end{tabular} 
    \text{ is not valid.}
    \]
    Explicitly, the columns $(4,3,5)$ and $(4,1,4)$ are not valid in the respective tableaux.
\end{remark}

\subsection{Linear quotient order on the elements of $(M_\a)^\ell$}\label{subsec:Order}
 
Let $A$ and $B$ be a pair of $3 \times \ell$ tableaux. Denote $(i_r,j_r,k_r)$ for the entries of the $r^{\rm th}$ column of $A$ and $(i'_r,j'_r,k'_r)$ for the $r^{\rm th}$ column of $B$. We write $(I_r, J_r, K_r)$ and $(I'_r, J'_r, K'_r)$ for the types of the $r^{\rm th}$ columns of $A$ and $B$ respectively. We define the type of the tableaux as follows:
\begin{itemize}
    \item $\type(A)=(I_1,\ldots,I_\ell;J_1,\ldots,J_\ell;K_1,\ldots,K_\ell)$,
    \item $\type(B)=(I'_1,\ldots,I'_\ell;J'_1,\ldots,J'_\ell;K'_1,\ldots,K'_\ell)$.
\end{itemize}
So, $\type(A)$ is a $3 \times \ell$ tableau whose $ij^{\rm th}$ entry indexes the part of the partition $\a$ that contains the $ij^{\rm th}$ entry of $A$.

We define the following ordering on the set of $3 \times \ell$ tableaux:
\[
A<_{\a}^\ell B\quad\Longleftrightarrow\quad\type(A)<\type(B)\quad\text{or}\quad\type(A)=\type(B)\quad\text{and}\quad A<^1B
\]
where
$\type(A)<\type(B)$ if and only if
\begin{itemize}
    \item There exists $t$ such that $k_\ell=k'_\ell,\ldots,k_{t+1}=k'_{t+1}$ and $k_t>k'_t$, or
    
    \item $(k_1,\ldots,k_\ell)=(k'_1,\ldots,k'_\ell)$ and there exists $t$ such that 
    $J_1=J'_1,\ldots,J_{t-1}=J'_{t-1}$ and $J_t<J'_t$, or
    
    \item $(k_1,\ldots,k_\ell)=(k'_1,\ldots,k'_\ell)$ and $(J_1,\ldots,J_\ell)=(J'_1,\ldots,J_\ell)$ and there exists $t$ such that 
    $I_\ell=I'_\ell,\ldots,I_{t+1}=I'_{t+1}$ and $I_t>I'_t$.
\end{itemize}

\medskip

For two tableaux $A$ and $B$ with $\type(A)=\type(B)$ we write that $A<^1B$ if and only if:
\begin{itemize}
    \item There exists some $t$ such that $i_1=i'_1,\ldots,i_{t-1}=i'_{t-1}$ and $i_t<i'_{t}$, or 
    
    \item $i_1=i'_1,\ldots,i_{\ell}=i'_{\ell}$ and there exists $t$ such that $j_1=j'_1,\ldots,j_{t-1}=j'_{t-1}$ and $j_t<j'_t$.
\end{itemize}
Note that if $A$ and $B$ have same type, then we know that $(k_1,\ldots,k_\ell)=(k'_1,\ldots,k'_\ell)$.

\medskip 
\noindent \textbf{Ordering generators.} We use the ordering of the $3 \times \ell$ tableaux to define an ordering on the generators of $(M_{\a})^\ell$. Let $m_1$ and $m_2$ be generators of $(M_{\a})^\ell$. For each $i \in [2]$, we write $T_i = \{A : m_A = m_i\}$ for the set of representative tableaux for $m_i$. We define
\[
    m_1 <_{\a}^\ell m_2 \iff \min(T_1) <_{\a}^\ell \min(T_2)
\]
where the minimum on the right hand side is taken with respect to $<_\a^{\ell}$: the ordering on the $3 \times \ell$ tableaux. We call the tableau $\min(T_1)$ the \textit{canonical type-representation} of $m_1$. By Lemma~\ref{lem:rep}, it follows that the canonical type-representation satisfies
\[
    \text{type}(\min(T_1)) = 
    \begin{bmatrix}
    I_1 & \le \dots \le & I_\ell \\
    J_1 & \dots & J_\ell \\
    k_1  & \le \dots \le & k_\ell
    \end{bmatrix}.
\]

\begin{example}
For the block diagonal matching field $\a=(2,2,2)$ with the partition sets $B=(12|34|56)$, every tableau of type $(3,3;1,1;3,3)$ is less than every tableau of type $(2,3;1,2;3,3)$ and $(2,3;1,2;2,3)$. Intuitively, we first order tableaux based on their types; all tableaux of the same type would appear consecutively and then we put an ordering on the tableaux of the same type (based on $<^1$). For example, the tableau $(3,5;2,3;6,6)<^1 (4,5;1,3;6,6)$ and they are of the same type $(2,3;1,2;3,3)$. 
\end{example}

\begin{theorem}\label{thm:linearquotient_powers}
The ideal $(M_\a)^\ell$ has a linear quotient property with respect to the ordering $<_\a^\ell$ for every block diagonal matching field $\a$ and every integer $\ell\geq 1$.
\end{theorem}
\begin{proof}
For each pair $(A,B)$ of monomials in $(M_\a)^\ell$, we consider their corresponding tableaux in canonical forms as described in Lemma~\ref{lem:rep}. Assume that $A <_\a^\ell B$. By working through the definition of the ordering $<_\a^\ell$
we construct another tableau $C$. This new tableau $C$ is formed by swapping an entry of $B$ with the first differing element of
$A$, with respect to the ordering $<_\a^\ell$. We show that the resulting tableau $C$ adheres to the constraints of a valid ordering derived from our matching field.
More precisely, let $A$ and $B$ be two monomials in $(M_\a)^\ell$ and write,
\[
A = 
\begin{bmatrix}
i_1 & \dots & i_\ell \\
j_1 & \dots & j_\ell \\
k_1 & \dots & k_\ell
\end{bmatrix} \, , \quad
B =
\begin{bmatrix}
i'_1 & \dots & i'_\ell \\
j'_1 & \dots & j'_\ell \\
k'_1 & \dots & k'_\ell
\end{bmatrix}\, .
\]
The tableaux are written with a canonical type-representation. So,  by Lemma~\ref{lem:rep}, we have:
\[
\begin{bmatrix}
I_1 & \le \dots \le & I_\ell \\
J_1 & \dots & J_\ell \\
k_1 & \le \dots \le &k_\ell
\end{bmatrix}
\text{ and }
\begin{bmatrix}
I'_1 &\le \dots \le & I'_\ell \\
J'_1 & \dots & J'_\ell \\
k'_1 &\le \dots \le & k'_\ell
\end{bmatrix} \, . 
\]

To show that $  (M_\a)^\ell$ has linear quotients, we prove that there exists a variable $v$ in $A$ such that $vB$ is a multiple of a generator $C$ of $  (M_\a)^\ell$ where $C <_\a^\ell B$. We find $v$ by taking cases on the different steps in the definition of $A<_\a^\ell B$.

\smallskip

\textbf{Case 1.} Suppose there exists $t$ such that $k_{\ell} = k'_{\ell}, \dots, k_{t+1} = k'_{t+1}$ and $k_t > k'_t$. Then by the ordering we have that,
\[
\begin{bmatrix}
i'_t \\
j'_t \\
k_t
\end{bmatrix}
<
\begin{bmatrix}
i'_t \\
j'_t \\
k'_t
\end{bmatrix} \, .
\]
So define $v := z_{k_t}$ and let $C'$ be the matrix obtained from $B$ by changing $k'_t$ with $k_t$ in column $t$. Since $B$ has a canonical type-representation and $k_\ell \ge k_{\ell - 1} \ge \dots \ge k_{t+1} \ge k_t > k'_t$. Let $C$ be the canonical type-representation of $C'$. Then, by definition, we have $C \le_\a^\ell C' <_\a^\ell B$.

\smallskip

\textbf{Case 2.} Suppose that $k_1 = k'_1, \dots, k_\ell = k'_\ell$ and there exists $t$ such that $J_1 = J'_1, \dots, J_{t-1} = J'_{t-1}$ and $J_t < J'_t$. Then by the ordering we have that,
\[
\begin{bmatrix}
i'_t \\
j_t \\
k'_t
\end{bmatrix}
<
\begin{bmatrix}
i'_t \\
j'_t \\
k'_t
\end{bmatrix} \, .
\]
Note that the left hand column is valid. We define $v := y_{j_t}$ and let $C'$ be the matrix obtained from $B$ by changing $j'_t$ to $j_t$ in column $t$. By assumption we have that the final rows of $A$ and $B$ are identical and $J_1 = J'_1 \le J_2 = J'_2 \le \dots \le J_t < J'_t$. Let $C$ be the canonical type-representation of $C'$. Then, by definition, we have $C \le_\a^\ell C' <_\a^\ell B$.

\smallskip

\textbf{Case 3.} Suppose that $k_1 = k'_1, \dots, k_\ell = k'_\ell$, $J_1 = J'_1, \dots, J_\ell = J'_\ell$ and there exists $t$ such that $I_\ell = I'_\ell, \dots, I_{t+1} = I'_{t+1}$ and $I_t > I'_t$. So by the ordering we have that,
\[
\begin{bmatrix}
i_t \\
j'_t \\
k'_t
\end{bmatrix}
<
\begin{bmatrix}
i'_t \\
j'_t \\
k'_t
\end{bmatrix} \, .
\]
In particular, note the left hand side is a valid column because
\[
J'_t = J_t \le I_t \le K_t = K'_t \text{ and } k_t = k'_t.
\]
So we have $j'_t < i_t < k'_t$. We define $v := x_{i_t}$ and let $C$ be the matrix obtained from $B$ by replacing $i'_t$ with $i_t$ in column $t$. Since $I_\ell = I'_\ell \ge \dots \ge I_{t+1} = I'_{t+1} \ge I_t > I'_t$ it follows that $C$ has canonical type-representation and $C <_\a^\ell B$.

\smallskip

\textbf{Case 4.} Suppose that $\text{type}(A) = \text{type}(B)$ and there exists $t$ such that $i_1 = i'_1, \dots, i_{t-1} = i'_{t-1}$ and $i_t < i'_t$. We show that,
\[
\begin{bmatrix}
i_t \\
j'_t \\
k'_t
\end{bmatrix}
<
\begin{bmatrix}
i'_t \\
j'_t \\
k'_t
\end{bmatrix} \, 
\]
and in particular, the left hand side is a valid column. 
We take cases on the types $I_t$ and $J_t$.
If $I_t = J_t$, then we have $i_t < i'_t < j'_t$, since $I_t = I'_t$ and $J_t = J'_t$. Hence, the left hand side is a valid column.
If $J_t < I_t$, then the left hand column is valid, since by $I_t = I'_t$ and $J_t = J'_t$.
So we define $v := x_{i_t}$ and let $C$ be the matrix obtained from $B$ by changing the entry $i'_t$ to $i_t$ in column $t$. Since $\text{type}(A) = \text{type}(B)$ it follows that $C$ is of the same type, so $C <_\a^\ell B$.

\smallskip

\textbf{Case 5.} Suppose that $\text{type}(A) = \text{type}(B)$, $i_1 = i'_1, \dots, i_\ell = i'_\ell$ and there exists $t$ such that $j_1 = j'_1, \dots, j_{t-1} = j'_{t-1}$ and $j_t < j'_t$. Then column $t$ of $A$ and $B$ respectively have the form,
\[
\begin{bmatrix}
i'_t \\
j_t \\
k'_t
\end{bmatrix}
<
\begin{bmatrix}
i'_t \\
j'_t \\
k'_t
\end{bmatrix} \, .
\]
So define $v := y_{j_t}$ and let $C$ be the matrix obtained from $B$ by changing $j'_t$ to $j_t$ in column $t$. Since $\text{type}(A) = \text{type}(B)$ it follows that $C$ is of the same type, so $C <_\a^\ell B$, as desired.
\end{proof}

\section{Powers of matching field ideals have cellular resolutions}\label{sec:cellular}
In this section, we demonstrate that the ideals $(M_\a)^\ell$ satisfy the conditions of Theorem~\ref{Theorem_cone}, thus possessing minimal cellular resolutions. Equivalently, their minimal free resolutions are supported on regular CW complexes.
By Theorem~\ref{thm:linearquotient_powers}, all the ideals $(M_\a)^\ell$ possess the linear quotient property. In this section, we will associate a decomposition map with the linear quotient order defined in the previous section, demonstrating that such a map possesses the desired properties required for Theorem~\ref{thm:linearquotient_powers}.

\begin{remark}
For $\ell = 1$, the decomposition map $b$ is uniquely determined since for any variable $v$ and monomial $m \in G(M)$ there is at most one monomial $m' \in G(M)$ such that $m' | v m$. However if $\ell > 1$ then for some variable $v$ and monomial $m \in G(  (M_\a)^\ell)$ there may exist $m'$ and $m'' \in G(  (M_\a)^\ell)$ such that $m' | v m$ and $m'' | vm$. For example, let ${\bf a} = [123|45]$ and $m = x_2 x_4 y_3 y_3 z_5 z_5 = (2,3,5)(4,3,5) \in G(M_{\bf a}^2)$. Then $x_1 \in \text{set}(m)$ however we have both $x_1 m = x_2(1,3,5)(4,3,5)$ and $x_1 m = x_4(2,3,5)(1,3,5)$. 
\end{remark} 

We choose the following decomposition function $b^\ell$ for $(M_\a)^\ell$.
Let $m \in G(M^\ell_{\bf a})$ then for any $i$ in $[n]$, where $i \in B_s$ for some $s$, we define:
\begin{align*}
    b^\ell(x_i m) &= \left\{
    \begin{tabular}{lcl}
        $m$                     & \quad & if $x_i \not\in \text{set}(m)$,  \\
        $\frac{m x_i}{x_j}$     & \quad & if $x_i \in \text{set}(m)$, where $j = \max\{k : \frac{m x_i}{x_k} \in G(M^\ell_{\bf a}) \}$,
    \end{tabular}
    \right. \\
    b^\ell(z_i m) &= \left\{
    \begin{tabular}{lcl}
        $m$                     & \quad & if $z_i \not\in \text{set}(m)$,  \\
        $\frac{m z_i}{z_j}$     & \quad & if $z_i \in \text{set}(m)$, where $j = \min\{k : z_k | m \}$,
    \end{tabular}
    \right. \\
    b^\ell(y_i m) &= \left\{
    \begin{tabular}{lcll}
        $m$                     & \quad & if $y_i \not\in \text{set}(m)$, & \\
        $\frac{m y_i}{y_j}$     & \quad & if $y_i \in \text{set}(m)$, where & if $\{k : k < i, k \in {B_s},  \frac{m y_i}{y_k} \in G(M^\ell_{\bf a}) \} \neq \emptyset$ then \\
        & & & $j = \min\{k : k < i, k \in {B_s}, \frac{m y_i}{y_k} \in G(M^\ell_{\bf a}) \}$ otherwise, \\
        & & & $j = \max\{k : k > i,  \frac{m y_i}{y_k} \in G(M^\ell_{\bf a}) \}$.
    \end{tabular}
    \right. \\
\end{align*}

With this decomposition function we now prove the following proposition.

\begin{proposition}\label{proposition:regular}
The decomposing map 
$b^\ell$ satisfies the exchange property ($*$) for each $\ell$. 
\end{proposition}

\begin{proof}
Throughout this proof we use $b$ to denote the map $b^\ell$. Fix a monomial $m \in G(M^\ell_{\bf a})$. We show that for any variables $v, w \in \text{set}(M^\ell_{\bf a})$ that condition $(*)$ holds by taking cases on $v, w$.

\textbf{Case 1.} Let $v = z_i$, $w = z_j$ and assume without loss of generality that $i < j$. We have $b(z_i m) = \frac{m}{z_i'}z_i$ where $i' = \min \{k : z_k | m\}$. There are now two cases,

\textbf{Case 1 (i).} There exists $k$ such that $z_k | b(z_i m)$ and $ k < i$. Let $j'$ be the smallest such $k$ and so $b(z_j b(z_i m)) = \frac{m z_i}{z_i' z_j'}z_j$. On the other hand we have $b(z_j m) = \frac{m}{z_i'}z_j$. Since $j' < i$, therefore $z_i \in \text{set}(b(z_j m))$ and so,
\[b(z_i b(z_j m)) = \frac{m z_j}{z_i' z_j'}z_i = b(z_j b(z_i m)).\]

\textbf{Case 1 (ii).} There does not exist $k$ such that $z_k | b(z_i m)$ and $k < i$. So $i = \min\{k : z_k | b(z_i m) \}$. Therefore $b(z_j b(z_i m)) = \frac{m}{z_i'} z_j$. On the other hand $b(z_j m) = \frac{m}{z_i'}z_j$. By assumption there are no $k$ such that $z_k | b(z_i m)$ and $k < i$. Since $j > i$ it follows that there exists no $k$ such that $z_k | b(z_j m)$ and $k < i$. Hence $z_i \not\in \text{set}(b(z_j m))$. Therefore,
\[
b(z_i b(z_j m)) = b(z_j m) = \frac{m}{z_i'}z_j = b(z_j b(z_i m)).
\]

\textbf{Case 2.} Let $v = z_i$ and $w = y_j$ or $x_j$. Let $w'$ be the variable such that $b(w m) = \frac{m}{w'}w$. So if $w = y_j$ then $w' = y_{j'}$ for some $j'$, and if $w = x_j$ then $w' = x_{j'}$ for some $j'$. Let $i' = \min\{k : z_k | m \}$. Since $z_i \in \text{set}(m)$, we have $i' < i$. So $\{k : z_k \in \text{set}(m) \} < \{k : z_k \in \text{set}(b(z_i m)) \}$. Therefore $\{x_k \in \text{set}(m) \} \subseteq \{x_k \in \text{set}(b(z_i m)) \} $ and $\{y_k \in \text{set}(m) \} \subseteq \{y_k \in \text{set}(b(z_i m)) \} $. If $T$ is a tableau representing $m$ such that swapping $w'$ with $w$ results in a valid tableau, then we can initially exchange $z_{i'}$ with $z_i$ to produce a tableau $T'$ representing the monomial $b(z_i m)$. The subsequent swapping of $w'$ with $w$ remains valid in $T'$ because by swapping $z_{i'}$ with $z_i$, we are augmenting the values in the last row of $T$, thereby preserving the validity of any swap within the first two rows.
Hence $b(w b(z_i m)) = \frac{m z_i}{z_i' w'} w$. On the other hand we have $b(w m) = \frac{m}{w'}w$. Since neither $w$ nor $w'$ are of the form $z_k$ for some $k$, we have $\{z_k : z_k | m\} = \{z_k : z_k | b(w m) \}$. Hence,
\[
b(z_i b(w m)) = \frac{m w}{w' z_{i'}}z_i = b(w b(z_i m)).
\]

\textbf{Case 3.} Let $v = x_i, w = x_j$ and assume without loss of generality that $i < j$. Let $i'$ and $j'$ be the values for which $b(x_i m) = \frac{m}{x_{i'}}x_i$ and $b(x_j m) = \frac{m}{x_{j'}}x_j$. We note that if $i, j \in I_s$ for some $s$ and $j' \in I_t$ where $t > s$, then there exists a tableau $T$ with a column $(j', b, c)$ such that $(j,b,c)$ is a valid column. It follows that $(i,b,c)$ is a valid column, hence $i' \ge j'$
We now take cases on $i, j, i', j'$.

\textbf{Case 3 (i).} Let $i, j, i', j' \in I_s$ for some $s$. First we show that $i' = j'$. If not then without loss of generality assume $i' > j'$. Let $T$ be a tableau corresponding to the monomial $m$. Suppose $(i', b, c)$ is a column in $T$. Since $j < j' < i'$ and $j, j', i' \in I_s$, it follows that $(j, b, c)$ is a valid column. By swapping the columns $(i', b, c)$ with $(j, b, c)$ in $T$, we deduce that $j' < i' \le \max\{k : \frac{m x_j}{x_k} \in G(M^\ell_{\bf a})\}$, contradicting the definition of $b^\ell$. So we have $i' = j'$.

Let $i'' = \max\{k \in I_s : x_k | b(x_j m)\}$ and note that $i'' \ge j$. So there are two cases.

\textbf{Case 3 (i) (a).} Let $i'' = j$. It follows that $b(x_i b(x_j m)) = \frac{m}{x_{j'}}x_i$. Consider $b(x_i m) = \frac{m}{x_{i'}}x_i = \frac{m}{x_{j'}}x_i$. Since $j > i$ therefore $\max\{k \in I_s : x_k | b(x_i m)\} \le j$ and so $x_j \not\in \text{set}(b(x_i m))$. And so,
\[
b(x_j b(x_i m)) = b(x_i m) = \frac{m}{x_{i'}}x_i = b(x_i b(x_jm)).
\]

\textbf{Case 3 (i) (b).} Let $i'' > j$. It follows that $b(x_i b(x_j m)) = \frac{m x_j}{x_{j'} x_{i''}}x_i$. On the other hand consider $b(x_i m) = \frac{m}{x_{i'}}x_i = \frac{m}{x_{j'}}x_i$. Since $i'' > j > i$, it follows that,
\[
b(x_j b(x_i m)) = \frac{m x_i}{x_{j'} x_{i''}}x_j = b(x_i b(x_j m)).
\]

\textbf{Case 3 (ii).} Let $i, j \in I_s$ for some $s$. Consider $b(x_i m)$. We assume that $i' =  j' \in I_t$ for some $t > s$. Likewise, the case where $i' > j'$ follows a similar line of reasoning. Suppose that there exists $j'' > i$ such that $b(x_j b(x_i m)) = \frac{m x_i}{x_{j'}x_{j''}}x_j$. Now consider $b(x_j m)$. Since $j'' > i$, any columns of the form $(j'', b, c)$ for which $(j,b,c)$ is a column, we must have $(i, b, c)$ is also a valid column. By definition of $b$, we take $j''$ as the maximum value for which we can swap $j''$ for $i$. Hence $b(x_i b(x_j m)) = b(x_j b(x_i m))$. 

Suppose now that there does not exist $j'' > i$ as above. Then it follows that $b(x_j b(x_i m)) = \frac{m}{x_{j'}}x_i$. Consider $b(x_j m)$, since $i < j$ and there does not exist $j''$ as above, we must have that $j = \max \{k > i : \frac{b(x_j m)}{x_k}x_i \in G(M^\ell_{\bf a}) \}$. So,
$
b(x_i b(x_j m)) = \frac{m}{x_{j'}}x_i = b(x_j b(x_i m)),
$ as desired.
\end{proof}

As an immediate corollary of Theorem~\ref{thm:linearquotient_powers}, Theorem~\ref{Theorem_cone} and Proposition~\ref{proposition:regular} we have that:
\begin{corollary}\label{cor: cell res}
Let $\a$ be any partition of $[n]$ and $M_\a$ be its corresponding matching field ideal. For every $\ell$, the ideal $(M_\a)^\ell$ has linear quotients. Furthermore, the minimal resolution of $(M_\a)^\ell$ obtained as an iterated mapping cone is cellular and supported on a regular $CW$-complex.
\end{corollary}

\begin{example}\label{example: gr24 1}
    Consider the matching field given by the partition $\a = (2,2)$. The generators of the ideal $M_{\a}$ are $
    x_1 y_2 z_3,\, 
    x_1 y_2 z_4,\,
    x_3 y_1 z_4,\,
    x_3 y_2 z_4.$ 
    The second power $M_{\a}^2$ is generated by monomials:
    \begin{align*}
    &x_{3}^{2}y_{2}^{2}z_{4}^{2},\,x_{1}x_{3}y_{2}^{2}z_{4}^{2},\,x_{1}^{2}y_{2}^{2}z_{4}^{2},\,x_{3}^{2}y_{1}y_{2}z_{4}^{2},\,x_{1}x_{3}y_{1}y_{2}z_{4}^{2},\\
    &x_{3}^{2}y_{1}^{2}z_{4}^{2},\,x_{1}x_{3}y_{2}^{2}z_{3}z_{4},\,x_{1}^{2}y_{2}^{2}z_{3}z_{4},\,x_{1}x_{3}y_{1}y_{2}z_{3}z_{4},\,x_{1}^{2}y_{2}^{2}z_{3}^{2}.
    \end{align*}
    The CW-complex of the cellular resolution of $M_{\a}^2$ given in Corollary~\ref{cor: cell res} is shown in Figure~\ref{fig:gr24 cell complex example}.

    \begin{figure}
        \centering
        \includegraphics{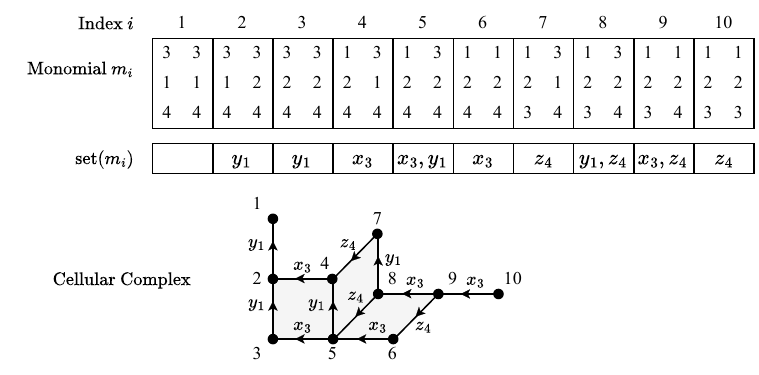}
        \caption{Cellular complex associated to the resolution of $(M_{\a})^2$ where $\a = (2,2)$}
        \label{fig:gr24 cell complex example}
    \end{figure}
\end{example}

\begin{example}\label{example: gr24 2}
    Consider the matching field given by the partition $\a = (1,3)$. The generators of the ideal $M_{\a}$ are $
    x_2 y_1 z_3,\, 
    x_2 y_1 z_4,\,
    x_3 y_1 z_4,\,
    x_2 y_3 z_4.$ 
    The second power $(M_{\a})^2$ is generated by:
    \begin{align*}
        &x_{2}^{2}y_{3}^{2}z_{4}^{2},\,
        x_{2}x_{3}y_{1}y_{3}z_{4}^{2},\,
        x_{2}^{2}y_{1}y_{3}z_{4}^{2},\,
        x_{3}^{2}y_{1}^{2}z_{4}^{2},\,
        x_{2}x_{3}y_{1}^{2}z_{4}^{2},\\
        &x_{2}^{2}y_{1}^{2}z_{4}^{2},\,
        x_{2}^{2}y_{1}y_{3}z_{3}z_{4},\,
        x_{2}x_{3}y_{1}^{2}z_{3}z_{4},\,
        x_{2}^{2}y_{1}^{2}z_{3}z_{4},\,
        x_{2}^{2}y_{1}^{2}z_{3}^{2}.
    \end{align*}
    The CW-complex of the cellular resolution of $(M_{\a})^2$ given in Corollary~\ref{cor: cell res} is shown in Figure~\ref{fig:gr24 cell complex example 2}.

    \begin{figure}
        \centering
        \includegraphics{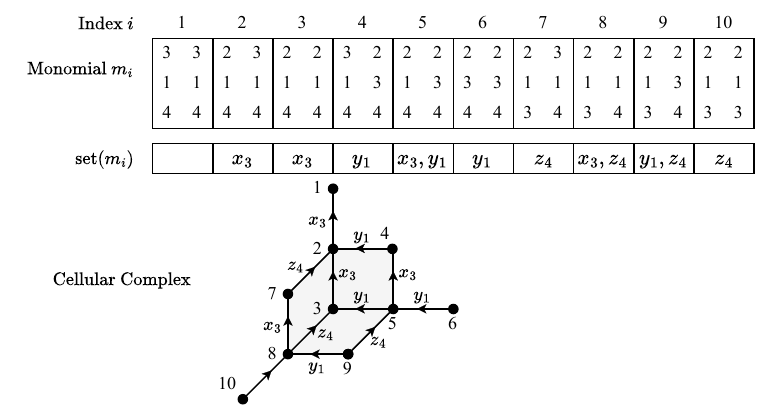}
        \caption{Cellular complex associated to the resolution of $(M_{\a})^2$ where $\a = (1,3)$}
        \label{fig:gr24 cell complex example 2}
    \end{figure}
\end{example}

\begin{remark}
  Examples~\ref{example: gr24 1} and \ref{example: gr24 2} represent monomial degenerations of the same determinantal ideal. They share the same Betti numbers, as reflected in the $f$-vector $(1,10,12,3)$ of the complex. However, despite this similarity, their CW-complexes differ. For instance, in Example~\ref{example: gr24 2}, the CW-complex features three 2-cells, each pair sharing an edge. In contrast, in Example~\ref{example: gr24 1}, there is a pair of two cells that do not share an edge.
\end{remark}

\begin{remark}
    Examples~\ref{example: gr24 1} and \ref{example: gr24 2} serve to demonstrate the definition of the monomial and decomposition function. To ensure that the examples and diagrams are not too long, the monomial ideals are small, i.e., they are generated by four monomials. As a result, these examples have cellular resolutions supported on paths, which coincide with those studied in \cite{cooper2022powers}. Consider the diagonal matching field $B_{(5)}$ whose associated monomials are $x_iy_jz_k$ with $1\leq i<j<k\leq 5$. We note that a cellular resolution of the associated monomial ideal $M_{(5)}$ is not supported on a graph, as its supporting CW-complex has $f$-vector $(1, 10, 15, 6)$. The second power of the matching field ideal $(M_{(5)})^2$ is minimally generated by fifty monomials and has a cellular resolution supported on a CW-complex with $f$-vector $(1, 50, 120, 105, 40)$. See also Example 5 and Figure 1 in \cite{Fatemeh2024} for the CW complex of the associated ideal of \( B_{(3,2)} \) for another such example.  
\end{remark}

\bibliographystyle{abbrv}
\bibliography{Trop.bib}

\bigskip
\bigskip
\noindent 
\footnotesize\small{\textbf{Authors' addresses}

\noindent
School of Mathematics, University of Edinburgh, Edinburgh, United Kingdom \\
E-mail address: {\tt oliver.clarke@ed.ac.uk}

\medskip
\noindent
Department of Computer Science, KU Leuven, Celestijnenlaan 200A, B-3001 Leuven, Belgium\\ 
   Department of Mathematics, KU Leuven, Celestijnenlaan 200B, B-3001 Leuven, Belgium\\  
   UiT – The Arctic University of Norway, 9037 Troms\o, Norway\\
   E-mail address: {\tt fatemeh.mohammadi@kuleuven.be}

\end{document}